\documentclass[reqno]{amsart}

\usepackage{amssymb}
\usepackage{graphicx}
\usepackage{amscd}
\usepackage[hidelinks]{hyperref}
\usepackage{color}
\usepackage{float}
\usepackage{graphics,amsmath,amssymb}
\usepackage{amsthm}
\usepackage{amsfonts}
\usepackage{latexsym}
\usepackage{epsf}
\usepackage{xifthen}
\usepackage{mathrsfs}
\usepackage{dsfont}
\usepackage{makecell}
\usepackage[FIGTOPCAP]{subfigure}
\usepackage{amsmath}
\allowdisplaybreaks[4]
\usepackage{listings}
\usepackage{etoolbox}
\usepackage{fancyhdr}
\usepackage{pdflscape}
\usepackage[title,toc,titletoc]{appendix}
\usepackage{enumitem}
\usepackage[noadjust]{cite}
\usepackage{multirow}
\usepackage{tikz}
\usetikzlibrary{decorations.markings} %for arrow in the middle
\usepackage{scalerel}
\usepackage{caption}
 \newtheoremstyle{mytheorem}% name % cf. thmtest.tex of AMSLaTeX
 {3pt}%      Space above
 {3pt}%      Space below
 {\slshape}% Body font
 {}%         Indent amount (empty = no indent,
 {\bfseries}% Thm head font
 {.}%        Punctuation after thm head
 { }%        Space after thm head: " " = normal interword space;
 {}%         Thm head spec (can be left empty, meaning `normal')

\numberwithin{equation}{section}

\newtheorem{theorem}{Theorem}[section]
\newtheorem*{theorem*}{Theorem}

\newtheorem{corollary}[theorem]{Corollary}
\newtheorem{lemma}[theorem]{Lemma}
\newtheorem{proposition}[theorem]{Proposition}

\newtheorem{innercustomgeneric}{\customgenericname}
\providecommand{\customgenericname}{}
\newcommand{\newcustomtheorem}[2]{%
	\newenvironment{#1}[1]
	{%
		\renewcommand\customgenericname{#2}%
		\renewcommand\theinnercustomgeneric{##1}%
		\innercustomgeneric
	}
	{\endinnercustomgeneric}
}
\newcustomtheorem{ctheorem}{Theorem}
\newcustomtheorem{clemma}{Lemma}

\theoremstyle{definition}
\newtheorem{definition}{Definition}[section]

\newtheorem{example}{Example}[section]
\newtheorem*{example*}{Example}
\newtheorem*{examples*}{Examples}
\newtheorem{remark}{Remark}[section]
\newtheorem*{remark*}{Remark}
\newtheorem*{remarks*}{Remarks}

\newtheoremstyle{named}{}{}{\itshape}{}{\bfseries}{.}{.5em}{#1\thmnote{ #3}}
\theoremstyle{named}

\newcommand{\abs}[1]{\left|#1\right|}
\newcommand{\set}[1]{\left\{#1\right\}}

\newcommand{\Keywords}[1]{\ifthenelse{\isempty{#1}}{}{\smallskip \smallskip \noindent \textbf{Keywords}. #1}}
\newcommand{\MSC}[2][2010]{\ifthenelse{\isempty{#2}}{}{\smallskip \smallskip \noindent \textbf{#1MSC}. #2}}
\newcommand{\abstractnote}[1]{\ifthenelse{\isempty{#1}}{}{\smallskip \smallskip \noindent \textsuperscript{\dag}#1}}

\makeatletter
\def\specialsection{\@startsection{section}{1}%
  \z@{\linespacing\@plus\linespacing}{.5\linespacing}%
  {\normalfont}}% NEW
\def\section{\@startsection{section}{1}%
  \z@{.7\linespacing\@plus\linespacing}{.5\linespacing}%
  {\normalfont\scshape}}% NEW
\patchcmd{\@settitle}{\uppercasenonmath\@title}{\Large\boldmath}{}{}
\patchcmd{\@settitle}{\begin{center}}{\begin{flushleft}}{}{}
\patchcmd{\@settitle}{\end{center}}{\end{flushleft}}{}{}
\patchcmd{\@setauthors}{\MakeUppercase}{\normalsize}{}{}
\patchcmd{\@setauthors}{\centering}{\raggedright}{}{}
\patchcmd{\section}{\scshape}{\large\bfseries\boldmath}{}{}
\patchcmd{\subsection}{\bfseries}{\bfseries\boldmath}{}{}
\renewcommand{\@secnumfont}{\bfseries}
\patchcmd{\@startsection}{\@afterindenttrue}{\@afterindentfalse}{}{}
\patchcmd{\abstract}{\leftmargin3pc}{\leftmargin1pc}{}{}

\def\maketitle{\par
  \@topnum\z@ % this prevents figures from falling at the top of page 1
  \@setcopyright
  \thispagestyle{empty}% this sets first page specifications
  \ifx\@empty\shortauthors \let\shortauthors\shorttitle
  \else \andify\shortauthors
  \fi
  \@maketitle@hook
  \begingroup
  \@maketitle
  \toks@\@xp{\shortauthors}\@temptokena\@xp{\shorttitle}%
  \toks4{\def\\{ \ignorespaces}}% defend against questionable usage
  \edef\@tempa{%
    \@nx\markboth{\the\toks4
      \@nx\MakeUppercase{\the\toks@}}{\the\@temptokena}}%
  \@tempa
  \endgroup
  \c@footnote\z@
  \@cleartopmattertags
}
\makeatother

\newcommand{\bC}{\mathbb{C}}
\newcommand{\bZ}{\mathbb{Z}}

\newcommand{\cT}{\mathcal{T}}
\newcommand{\cB}{\mathcal{B}}
\newcommand{\curly}{\mathrel{\leadsto}}
\def\Comp{\mathsf{Comp}}
\def\CComp{\mathsf{CComp}}
\def\DComp{\mathsf{DComp}}
\def\UComp{\mathsf{UComp}}
\def\nonu{\mathsf{nu}}
\def\leaf{\mathsf{leaf}}
\def\gleaf{\mathsf{gleaf}}
\def\int{\mathsf{int}}
\def\pint{\mathsf{pint}}
\def\Nuo{\mathsf{Nuo}}
\def\Onu{\mathsf{Onu}}
\def\Dnu{\mathsf{Dnu}}
\def\Do{\mathsf{Do}}
\def\tonu{\widetilde{\Onu}}
\def\nuo{\mathsf{nuo}}
\def\onu{\mathsf{onu}}
\def\dnu{\mathsf{dnu}}
\newcommand{\sdo}{\mathsf{do}}
\def\ap{\mathsf{ap}}
\def\uf{\underline{f}}
\def\Orb{\mathsf{Orb}}

\title[]{A group action on cyclic compositions and $\gamma$-positivity}

\author[S. Fu]{Shishuo Fu}
\address[S. Fu]{College of Mathematics and Statistics, Chongqing University, Chongqing 401331, P.R. China}
\email{fsshuo@cqu.edu.cn}

\author[J. Yang]{Jie Yang}
\address[J. Yang]{College of Mathematics and Statistics, Chongqing University, Chongqing 401331, P.R. China}
\email{18903845128@163.com}

\date{}

\begin{document}

\begin{abstract}
	
	Let $w_{n,k,m}$ be the number of Dyck paths of semilength $n$ with $k$ occurrences of $UD$ and $m$ occurrences of $UUD$. We establish in two ways a new interpretation of the numbers $w_{n,k,m}$ in terms of plane trees and internal nodes. The first way builds on a new characterization of plane trees that involves cyclic compositions. The second proof utilizes a known interpretation of $w_{n,k,m}$ in terms of plane trees and leaves, and a recent involution on plane trees constructed by Li, Lin, and Zhao. Moreover, a group action on the set of cyclic compositions (or equivalently, $2$-dominant compositions) is introduced, which amounts to give a combinatorial proof of the $\gamma$-positivity of the Narayana polynomial, as well as the $\gamma$-positivity of the polynomial $W_{2k+1,k}(t):=\sum_{1\le m\le k}w_{2k+1,k,m}t^m$ previously obtained by B\'{o}na et al, with apparently new combinatorial interpretations of their $\gamma$-coefficients.


	\MSC{05A05, 05A10, 05A15, 05C05}	

\end{abstract}

\maketitle

\section{Introduction}\label{sec:intro}
The sequence of {\it Catalan numbers}: $1,1,2,5,14,42,132,\ldots,$ is one of the most ubiquitous number sequences in enumerative combinatorics; see for example \cite[Ex.~6.19]{Sta1999} and \cite{Sta2015}. Two of the most familiar combinatorial models enumerated by Catalan numbers are {\it Dyck paths} and {\it plane trees}, whose definitions will be briefly recalled in section~\ref{sec:notation}. If one counts Dyck paths by their number of peaks, or equivalently plane trees by their number of leaves, a well-studied refinement of Catalan numbers arises, namely the {\it Narayana numbers} \cite[Chapter 2.3]{Pet2015}, which can be explicitly computed as
$$N_{n,k}=\frac{1}{n}\binom{n}{k}\binom{n}{k-1},$$
for $1\le k\le n$. Now if we write each Dyck path $p$ of semilength $n$ as a word $w(p)$ consisting of $n$ letters of $U$ (for up-step) and $n$ letters of $D$ (for down-step), then the number of peaks of $p$ is precisely the number of $UD$-factors contained in $w(p)$. Quite recently, B\'ona et al. \cite{BDL2022} initiated the further refined counting of Dyck paths according to the number of $UUD$-factors; see also \cite[Lemma~2.5]{LK2021} for a previous study on the number of $UUD$-factors over a Dyck path $p$ under the name $\mathsf{segm}(p)$. 

B\'ona et al. introduced
\begin{align}\label{formula for w}
w_{n,k,m} &= \begin{cases}
\frac{1}{k}\binom{n}{k-1}\binom{n-k-1}{m-1}\binom{k}{m}, & \text{if $m>0$, $m\le k$, and $k+m\le n$,}\\
1, & \text{if $m=0$ and $n=k$,}\\
0, & \text{otherwise,}
\end{cases}
\end{align}
and they showed in \cite[Thm.~1.2]{BDL2022} that the number of Dyck paths of semilength $n$ with $k$ $UD$-factors and $m$ $UUD$-factors is given by $w_{n,k,m}$. They went on to investigate more properties of these numbers in the special cases of $n=2k+1$ and $n=2k-1$, such as {\it symmetry} and {\it $\gamma$-positivity} (see section~\ref{sec:gamma}), by noting the connection with Narayana numbers
\begin{align}\label{id:w-Narayana}
w_{2k+1,k,m} &= \binom{2k+1}{k-1}N_{k,m},\\
w_{2k-1,k,m} &= \binom{2k-1}{k-1}N_{k-1,m}.
\label{id:w-Narayana2}
\end{align}
Motivated by their work, we shall consider in this paper a lesser known witness of Catalan numbers, namely the set of cyclic compositions.
 % {\bf [Todo 2]} It plays certain role in the study of $3+1$-conjecture, cite Wang's work.

Let us start with the notion of composition. A sequence $c=(c_1,c_2,\ldots,c_k)$ of positive integers is said to be a {\it composition} of $n$, if $c_1+c_2+\cdots+c_k=n$. Each $c_i$ is called a {\it part} of $c$ and we use $c\vDash n$ to indicate that $c$ is a composition of $n$. Let $\Comp_{n,k}$ denote the set of all compositions of $n$ with $k$ parts. It is well known that $\abs{\cup_{1\le k\le n}\Comp_{n,k}}=2^{n-1}$ for $n\ge 1$, the total number of compositions of $n$, and $\abs{\Comp_{n,k}}=\binom{n-1}{k-1}$. For our purposes, we also consider a further refinement $\Comp_{n,k,m}$, the set of compositions of $n$ into $k$ parts wherein exactly $m$ parts are larger than $1$. We refer to these parts as ``non-unitary'' in the sequel. Clearly $\Comp_{n,k}=\cup_{m=0}^k\Comp_{n,k,m}$. Given a sequence $s=s_1s_2\cdots s_n$, we say that a sequence $s'$ is a {\it cyclic shift} of $s$ if $s'$ is of the form
$$s'=s_is_{i+1}\cdots s_ns_1s_2\cdots s_{i-1}$$
for some $1\le i\le n$, where $s_{n+1}$ is understood to be $s_1$. Clearly there are $n$ cyclic shifts of $s$ including itself, although they are not necessarily all distinct. We denote $s\sim s'$ whenever $s$ and $s'$ are cyclic shifts of each other. We define a {\it cyclic composition} $[c]$ to be the equivalence class of a composition $c$ under cyclic shifting. Since cyclic shifting preserves the number of parts and the size of each part, it makes sense to introduce $\CComp_{n,k}$, the set of cyclic compositions of $n$ into $k$ parts, as well as its refinement $\CComp_{n,k,m}$. Realizing cyclic compositions as the underlying structure, we are naturally led to the following new interpretation of $w_{n,k,m}$ in terms of plane trees, which can be viewed as our first main result.
\begin{theorem}\label{thm:tree-2}
The number of plane trees with $n$ edges, $k$ internal nodes, and $m$ internal nodes with degree larger than one is given by $w_{n,k,m}$.
\end{theorem}

We supply two proofs of Theorem~\ref{thm:tree-2} in section~\ref{sec:pf of tree-2}. The first proof is based on a useful representation of plane trees in terms of a cyclic composition jointly with a subset of non-root vertices. This representation outputs directly the formula given in \eqref{formula for w}; the second proof is to apply a recent involution on plane trees constructed by Li, Lin and Zhao~\cite{LLZ2024}. Our first approach using cyclic compositions also affords us with a model that we believe is more amenable to the consideration of a group action, and consequently leads to a new combinatorial proof of the $\gamma$-positivity for Narayana numbers as well as new combinatorial interpretations of their $\gamma$-coefficients. These are the main results of the paper and constitute our section~\ref{sec:gamma}.

%%%%%%%%%%%%%%%%%%%%%%%%%%%%%%%%%%%%%%%%%%%%%%%%%%%%%%
\section{Notations and Preliminary results}\label{sec:notation}

We recall the definitions of Dyck paths and plane trees and collect some known results from \cite{BDL2022} that are relevant to this paper.

A Dyck path of semilength $n$ is a path in $\bZ^2$ with two types of allowed steps, namely $(1,1)$ and $(1,-1)$, that starts at the origin $(0,0)$, ends at $(2n,0)$, and never goes below the horizontal axis. As mentioned in the introduction, we use interchangeably a $UD$-word to represent a Dyck path. Using generating function techniques, B\'ona et al. derived the following result.
\begin{theorem}[{B\'ona et al.~\cite[Thm.~1.2]{BDL2022}}]\label{thm:path}
The number of Dyck paths of semilength $n$ with $k$ $UD$-factors and $m$ $UUD$-factors is given by $w_{n,k,m}$.
\end{theorem}
In the course of constructing a combinatorial proof of the symmetry $w_{2k+1,k,m}=w_{2k+1,k,k+1-m}$ for all $1\le m\le k$, they utilized the following alternative interpretation of $w_{n,k,m}$ in terms of plane trees. A {\it tree} is a simply-connected acyclic graph. All non-empty trees considered in this paper have one designated vertex called the {\it root}. Starting from each vertex $v$ of a tree $T$, there is a unique (shortest) path tracing back to the root of $T$, and the closet node to $v$ on this path, say $u$, is referred to as the parent of $v$, and $v$ is called a child of $u$. A plane tree is a tree where all the children of a given vertex are assigned a certain order from left to right. The number of children adjacent to a node $v$ is called the degree of $v$. A vertex is called a {\it leaf} if it has degree zero (i.e., no child), otherwise it is said to be an {\it internal node}. A leaf $v$ is a {\it good leaf} if $v$ is the left-most child of a non-root vertex. Now we can state the second interpretation of $w_{n,k,m}$.

\begin{theorem}[{B\'ona et al.~\cite[Prop.~3.1]{BDL2022}}]\label{thm:tree-1}
The number of plane trees with $n$ non-root vertices, $k$ leaves, and $m$ good leaves is given by $w_{n,k,m}$.
\end{theorem}

We remark that theorem~\ref{thm:tree-1} follows immediately from theorem~\ref{thm:path} by applying a well-known bijection denoted $\theta$ from plane trees to Dyck paths. Given a plane tree $T$ with $n$ edges, we traverse its edges in preorder so that each edge gets passed by twice. (I.e., we start with the edge connecting the root with its leftmost child, then traverse this entire subtree, trace this edge again to get back to the root, then continue to the next subtree immediately to the right, so on and so forth.) To each edge passed on the way for the first time there corresponds a step $U$, and to each edge passed on the way for the second time there corresponds a step $D$. This gives us a Dyck path $\theta(T)$ of semilength $n$. For an example of this bijection, the Dyck path and the plane tree in Figure~\ref{fig:Dyck-tree} are connected by $\theta$.
% {\bf [Todo 3]} Add two figures, one of Dyck path, one of plane tree, connected by the map $\Phi$.

\begin{figure}[!ht]
	\centering
	\subfigure{
		\begin{minipage}[b]{1\textwidth}
			\centering
			\begin{tikzpicture}[scale=0.4]
				\draw[step=1,help lines] ( 0,0 ) grid ( 14, 4);
				\draw[black, ultra thick](0,0)--(3,3);
				\draw[black, ultra thick](3,3)--(5,1);
				\draw[black, ultra thick](5,1)--(6,2);
				\draw[black, ultra thick](6,2)--(7,1);
				\draw[black, ultra thick](7,1)--(9,3);
				\draw[black, ultra thick](9,3)--(10,2);
				\draw[black, ultra thick](10,2)--(11,3);
				\draw[black, ultra thick](11,3)--(14,0);
			\end{tikzpicture}
			\caption*{Dyck path of semilength 7 with 4 $UD$-factors and 2 $UUD$-factors.}
		\end{minipage}
	}
	\hfill
	\subfigure{
		\begin{minipage}[b]{1\textwidth}
			\centering
			\begin{tikzpicture}[scale=0.7]
				\draw[-] (0,0) to (0,1);
				\draw[-] (0,1) to (1,2);
				\draw[-] (1,2) to (1,3);
				\draw[-] (1,2) to (1,1);
				\draw[-] (1,2) to (2,1);
				\draw[-] (2,1) to (1,0);
				\draw[-] (2,1) to (3,0);
				\node at (0,0){$\bullet$};
				\node at (0,1){$\bullet$};
				\node at (1,2){$\bullet$};
				\node at (1,3){$\bullet$};
				\node at (1,1){$\bullet$};
				\node at (2,1){$\bullet$};
				\node at (1,0){$\bullet$};
				\node at (3,0){$\bullet$};
				\draw[color=red]  (0,0) circle (0.3);
				\draw[color=red]  (1,0) circle (0.3);
			\end{tikzpicture} 
			\caption*{Plane tree on 7 non-root vertices with 4 leaves and 2 good leaves (circled in red).}
		\end{minipage}
	}
	\caption{A Dyck path and its corresponding plane tree under the map $\theta$}\label{fig:Dyck-tree}
\end{figure}

Note that our theorem~\ref{thm:tree-2} supplies a third combinatorial interpretation of $w_{n,k,m}$, again in terms of plane trees but with the role played by leaves replaced by internal nodes. So theorem~\ref{thm:tree-1} and theorem~\ref{thm:tree-2} are in some sense dual to each other. Indeed, our second proof of theorem~\ref{thm:tree-2} utilizes a reflection-like involution due to Li-Lin-Zhao~\cite{LLZ2024}, thereby illustrats such a duality.

Next, we introduce further notations and collect some enumerative results concerning cyclic compositions and their representatives called ($k$-)dominant compositions.

Given a positive integer $k$ and a sequence $s=s_1s_2\cdots s_l$ composed of only $U$s and $D$s, we say that $s$ is {\it $k$-dominating} if for every $1\le i\le l$, the prefix $s_1s_2\cdots s_i$ of $s$ has more copies of $U$ than $k$ times the number of copies of $D$. Furthermore, each composition $c=(c_1,c_2,\ldots,c_l)$ corresponds naturally to a unique $UD$-sequence, namely
$$\tau(c):=U^{c_1}DU^{c_2}D\cdots U^{c_l}D.$$
We say a composition $c$ is {\it $k$-dominating} whenever its image under $\tau$ is a $k$-dominating $UD$-sequence. Alternatively, we have the following characterization of $k$-dominance for compositions, which immediately follows from the definitions of the map $\tau$ and the original $k$-dominance for $UD$-sequences.

 % In particular, we see that a path $p$ is a Dyck path if and only if its corresponding $UD$-word $w(p)$, when appended from left a letter $U$ as $Uw(p)$, becomes $1$-dominant. From this perspective, $k$-dominating $UD$-sequences naturally generalize Dyck paths.

For notational convenience, we introduce an operator $f_k$. For any composition $c=(c_1,c_2,\ldots,c_l)$, we let $f_k(c;i,j):=\sum_{t=i}^j (c_t-k)$, $1\le i\le j\le l$, then we have $f_k(c;i,j)=\sum_{t=i}^{j} f_k(c;t,t)$. In particular, we simply write $f$ for the operator $f_2$. And likewise, {\it dominating} simply means $2$-dominating, unless otherwise noted.

\begin{proposition}\label{char:k-domi}
A composition $c=(c_1,c_2,\ldots,c_l)$ is $k$-dominating, if and only if for each $1\le i\le l$, we have $f_k(c;1,i)>0$.
\end{proposition}

Next, we would like to pick out a unique representative for each cyclic composition $[c]\in\CComp_{2n+1,n}$, which turns out to be convenient for applying our group action in section~\ref{sec:gamma}. Our choice is to take the dominating ones, and this is justified by the following cycle lemma.

\begin{lemma}[Cycle lemma~\cite{DM1947}] 
Let $k$ be a positive integer. For any sequence $s=s_1s_2\cdots s_{n+m}$ consisting of $m$ copies of $U$ and $n$ copies of $D$, there are exactly $m-kn$ cyclic shifts of $s$ that are $k$-dominating.
\end{lemma}

\begin{definition}
For each positive integer $n$, we let $\DComp_n$ be the set of ($2$-)dominating compositions in $\Comp_{2n+1,n}$. Similarly, for each $1\le m\le n$, we denote $\DComp_{n,m}$ the set of dominating compositions in $\Comp_{2n+1,n,m}$.
\end{definition}

Applying the cycle lemma and with the map $\tau$ in mind, we see that each cyclic composition $[c]\in\CComp_{2n+1,n}$ has precisely one cyclic shift, say $c'$, that belongs to $\DComp_n$. In particular, we see that $|\DComp_n|=|\CComp_{2n+1,n}|$. For example, among the four choices contained in the class
$$[(2,3,1,3)]=\{(2,3,1,3),(3,1,3,2),(1,3,2,3),(3,2,3,1)\}\in\CComp_{9,4},$$
the only dominating composition is $(3,2,3,1)$. Moreover, relying on the fact that $n$ and $2n+1$ are coprime with each other, we see all $n$ cyclic shifts of a given composition $c\in\Comp_{2n+1,n}$ are distinct, hence
\begin{align*} %\label{dc-cc}
\abs{\DComp_n}=\abs{\CComp_{2n+1,n}}=\frac{1}{n}\abs{\Comp_{2n+1,n}}=\frac{1}{n}\binom{2n}{n-1}=\frac{1}{n+1}\binom{2n}{n},
\end{align*}
rendering $\set{\DComp_n}_{n\ge 0}$ a Catalan family. (I.e., combinatorial objects that are enumerated by Catalan numbers).

Next we introduce the notions of smoothness and L-smoothness for any given integer sequence, as defined by Mansour and Shattuck~\cite{MS2023}.
\begin{definition}
An integer sequence $w=w_1\ldots w_n$ is said to be {\it smooth} if $|w_{i+1}-w_i|\le 1$ for each $1\le i\le n-1$. Moreover, it is said to be {\it L-smooth} if the weaker condition $w_{i+1}-w_i\ge -1$ holds for all $1\le i\le n-1$.
\end{definition}
The following two propositions become handy for our later arguments in section~\ref{sec:gamma} that utilize the L-smoothness.
\begin{proposition}\label{down-by-one}
Suppose $w=w_1\ldots w_n$ is an L-smooth sequence of integers, and $w_i>w_j$ for certain $1\le i<j\le n$, then there exists an index $k$, $i\le k<j$, such that $w_k=w_i$ and $w_{k+1}=w_k-1$.
\end{proposition}
\begin{proof}
We define the set $S_{i,j}:=\set{k:i\le k<j,~w_k=w_i}$. Clearly $i\in S_{i,j}$, so this set is non-empty, and we can take $m:=\max{S_{i,j}}$. One verifies that $w_{m+1}=w_m-1$, making $m$ a qualified index. Indeed, if $w_{m+1}>w_m=w_i>w_j$, then there is no way to go from $w_{m+1}$ down to $w_j$ without hitting the value $w_m$, since $w$ is L-smooth. So we must have $w_{m+1}<w_m$, and again L-smoothness forces that $w_{m+1}=w_m-1$, as desired.
\end{proof}

\begin{proposition}\label{L-smooth}
For any composition $c=(c_1, c_2, \ldots, c_n)$, the sequence $f(c;1,1),\\f(c;1,2),\ldots,f(c;1,n)$ is L-smooth. In particular, for every $1\le i\le n$, $f(c;1,i)-f(c;1,i-1)=-1$ if and only if $c_i=1$, where we set $f(c;1,0)=0$ as convention.
\end{proposition}
\begin{proof}
A direct computation according to the definition of the operator $f$ (=$f_2$) shows that for any $1\le i\le n$, the increment $f(c;1,i)-f(c;1,i-1)=f(c;i,i)=c_i-2\ge -1$. And the equal sign holds if and only if $c_i=1$, as claimed.
\end{proof}

As already noted by B\'ona et al., one can associate each dominating composition $c\in\DComp_{n,m}$ bijectively with a Dyck path of semilength $n$ and $m$ peaks, thus showing that (see~\cite[Lemma~4.1]{BDL2022})
\begin{align}\label{Dcomp is Nara}
|\DComp_{n,m}|=N_{n,m},
\end{align}
the Narayana number. To keep this paper self-contained and to make dominant compositions the truly focal point, here we rederive \eqref{Dcomp is Nara} by taking a more direct approach to consider the {\it Narayana polynomial} $N_n(t):=\sum_{1\le m\le n}N_{n,m}t^m$, which is known (see \cite[Sect.~2.3]{Pet2015}) to satisfy the following recurrence. Let $N_0(t):=1$, we have $N_1(t)=t$ and for $n\ge 2$,
\begin{align}\label{rec:Narapoly}
N_n(t) &=tN_{n-1}(t)+\sum_{i=0}^{n-2} N_i(t)N_{n-1-i}(t).
\end{align}

% \noindent {\bf [Todo 4]:} Prove directly that $|\DComp_{k,m}|=N_{k,m}$ is the Narayana number, without using Lemma 4.1 from \cite{BDL2022}.

% Narayana number is a refinement of Catalan number. Its generating function satisfies  (see \cite{Pet2015}). Next we will give a new combinatorial interpretation of Narayana number. 

Let $\nonu(c)$ be the number of non-unitary parts of $c$. We define $C_0(t):=1$ and $C_n(t):= {\textstyle \sum_{c\in \DComp_n}t^{\nonu(c)}}$, then comparing the following result with \eqref{rec:Narapoly} immediately yields \eqref{Dcomp is Nara} and justifies the pair $(\DComp_n ,\nonu)$ as a new witness of Narayana polynomial $N_n(t)$.
\begin{proposition}
	For $n\ge 2$, there holds $$C_n(t)=tC_{n-1}(t)+C_{n-1}(t)+\sum_{i=1}^{n-2}C_i(t)C_{n-1-i}(t).	$$
\end{proposition}
\begin{proof}
Given a $c=\left(c_1, c_2, \ldots, c_n\right)\in \DComp_{n}$, its last entry can only be $1$ or $2$. In fact if $c_n\ge 3$, then $f(c;1,n)=f(c;1,n-1)+c_n-2\ge 2$, which contradicts with $f(c;1,n)=2n+1-2n=1$. We consider the following three cases:
\begin{itemize}
\item If $c_n=2$, then removing it from $c$ gives us $(c_1, c_2, \ldots, c_{n-1})\in\DComp_{n-1}$.
\item If $c_n=1$ and for every $1\le i\le n-1$, we have $f(c;1,i)>1$, then removing $c_n$ and subtracting $c_1$ by $1$ we get $(c_1-1, c_2, \ldots, c_{n-1})\in\DComp_{n-1}$.
\item Otherwise $c_n=1$, and there exists a certain $j$, $1\le j\le n-1$, such that $f(c;1,j)=1$. Let $m$ be the largest such $j$, then one checks that both $a:=(c_1, c_2,\ldots, c_m)$ and $b:=(c_{m+1}, c_{m+2},\ldots, c_{n-1})$ are dominant compositions. Note the maximality of $m$ is needed in showing the dominance of composition $b$.
\end{itemize}

For the first case, we can append $2$, a non-unitary part, to the right of a given $c\in \DComp_{n-1}$, and recovers uniquely a dominant composition in $\DComp_n$, so this case explains the term $tC_{n-1}(t)$.
	
The second case is seen to be revertible by a similar argument. Namely, given a composition from $\DComp_{n-1}$, we append $1$ to its right and increase the first part by $1$, to uniquely recover a dominant composition satisfying the condition of case 2. This corresponds to the term $C_{n-1}(t)$.
	
Finally for the last case, we concatenate two non-empty dominant compositions of length $i$ and $n-1-i$ respectively to get a new composition, and append $1$ to its end. This gives us a dominant composition of length $n$ and explains the summation in the recurrence.
\end{proof}
% So  But this one thing needs to be noted that there is a difference of $t$ between our definition of $|\DComp_k|$ and the traditional  Narayana polynomial, because our definition of composition contains at least one element greater than $1$ as long as it exists. 

%Let $\nonu(c)$ be the number of non-unitary parts of $c$, then we have the following interpretation of the $n$-th {\it Narayana polynomial} $N_n(t)$. $N_0(t):=1$, and for $n\ge 1$,
%\begin{align}
%N_n(t):=\sum_{m=1}^n N_{n,m}t^m=\sum_{c\in\DComp_n}t^{\nonu(c)}.
%\end{align}
% \begin{lemma}[B\'ona et al]
% For $n\ge 1$,
% \begin{align*}
% w_{n,k,m} & :=\#\left\{p \in \operatorname{Dyck}_n: \operatorname{peak}(p)=k, p \text { has } m\ UUD \text {-factors }\right\} \\ & =\#\left\{T \in \mathcal{T}_n: \operatorname{leaf}(T)=k,\ T \text { has } m \text { good leaves }\right\}.
% \end{align*}
% \end{lemma}
% B\'ona and the sixth author \cite{BDL2022} showed that $w_{2k+1,k,m} = w_{2k+1,k,k+1-m} $,  $w_{n,k,m}=\frac{1}{k}\binom{n}{k-1}\abs{\Comp_{n,k,m}}$, and in particular, $w_{2k+1,k,m}=\binom{2k+1}{k-1}N_{k,m}$.

We end this preliminary section with a direct enumeration of $\Comp_{n,k,m}$.
\begin{proposition}\label{prop:comp}
For $1\le m\le k <n$, we have
\begin{align}\label{id:Compnkm}
|\Comp_{n,k,m}|=\binom{n-k-1}{m-1}\binom{k}{m}.
\end{align}
In particular, the following Chu-Vandermonde identity holds. 
$$\sum_{m=1}^k\binom{n-k-1}{m-1}\binom{k}{m}=\binom{n-1}{k-1},$$
since $\Comp_{n,k}=\cup_{1\le m\le k}\Comp_{n,k,m}$ for $n>k\ge 1$.
\end{proposition}
\begin{proof}
Each composition $c\in\Comp_{n,k,m}$ gives a solution to the Diophantine equation
\begin{align}\label{dio1}
x_1+x_2+\cdots+x_k=n,
\end{align}
with all $x_i\ge 1$ and precisely $m$ of them are strictly larger than $1$. Say these are $x_{i_1},x_{i_2},\ldots,x_{i_m}$. We subtract all $1$'s to the right of \eqref{dio1}, make the change of variables $y_j:=x_{i_j}-2$ and record the resulting equation
\begin{align}\label{dio2}
y_1+y_2+\cdots+y_m=n-k-m.
\end{align}
Note that now each $y_j\ge 0$ and \eqref{dio2} has $\binom{n-k-1}{m-1}$ non-negative solutions. Finally, each non-negative solution of \eqref{dio2} gives rise to $\binom{k}{m}$ qualified solutions of \eqref{dio1} since we have to determine where to place those $k-m$ deleted variables $x_i=1$, besides recovering $x_{i_j}$ from $y_j$.
\end{proof}

%%%%%%%%%%%%%%%%%%%%%%%%%%%%%%%%%%%%%%%%%%%%%%%%%%%%%%%
\section{Two proofs of Theorem~\ref{thm:tree-2}}\label{sec:pf of tree-2}
We abuse the notation a bit to denote $\binom{S}{i}$ the set of all $i$-element subsets of a given set $S$. The first proof of theorem~\ref{thm:tree-2} contains two steps. The first step is to enumerate all the pairs $(A,c)$, where $c\in\Comp_{n,k,m}$ and $A\in\binom{[n]}{k-1}$ with $[n]:=\{1,2,\ldots,n\}$. Denote the set of such pairs by $P_{n,k,m}$. For an internal node from a plane tree, we say it is unitary if its degree is one, otherwise it is non-unitary. Let $\cT_{n,k,m}$ be the set of plane trees with $n$ edges, $k$ internal nodes, $m$ of which are non-unitary. We are goint to construct in Theorem~\ref{thm:bij-tree-pair} a $k$-to-$1$ mapping $\phi$ between $P_{n,k,m}$ and $\cT_{n,k,m}$ for generic values of $n,k,m$, i.e., for $1\le m\le k$ and $k+m\le n$. Note that when $m=0$ and $k=n$, $\cT_{n,n,0}$ contains only one plane tree, namely the $n$-chain, and $w_{n,n,0}=1$ as well. For other values of $n,k,m$, $\cT_{n,k,m}$ is empty and $w_{n,k,m}=0$. So theorem~\ref{thm:tree-2} indeed follows from the next lemma and theorem~\ref{thm:bij-tree-pair}.
\begin{lemma}\label{lem:pair enum}
For $1 \leq m \leq k$ and $k+m\le n$, we have
\begin{align}\label{pt-p}
\abs{P_{n,k,m}} &= \binom{n}{k-1} \binom{n-k-1}{m-1}\binom{k}{m}.
\end{align}
\end{lemma}
\begin{proof}
This is a direct consequence of the product rule and Eq.~\eqref{id:Compnkm}.
\end{proof}

Now, before we construct the aforementioned $k$-to-$1$ mapping $\phi$, let us introduce a convenient way of representing a given pair $(A,c)\in P_{n,k,m}$, where $A=\set{a_1,\ldots,a_{k-1}}\in[n]$ and $c=(c_1,c_2,\ldots,c_k)$. We start by listing out $1,2,\ldots,n$, with a bar inserted between $\sum_{1\le i\le j}c_i$ and $1+\sum_{1\le i\le j}c_i$, for each $1\le j\le k-1$. Then we underline those numbers that occur in $A$, i.e., $a_1,a_2,\ldots,a_{k-1}$. For instance, the pair $(\set{3,4,6},(2,1,2,1))$ is expressed as 
$$12\mid \underline{3}\mid \underline{4} 5\mid \underline{6}.$$
We call such an expression an {\it underlined composition}. The set of all underlined compositions where the associated composition belongs to $\Comp_{n,k,m}$ is denoted as $\UComp_{n,k,m}$. It should be clear how to uniquely recover a pair $(A,c)$ from an underlined composition, hence we see that $\abs{\UComp_{n,k,m}}=\abs{P_{n,k,m}}$. From now on we shall speak of the pair $(A,c)$ and its corresponding underlined composition interchangeably.

\begin{theorem}\label{thm:bij-tree-pair}
For $1 \leq m \leq k$ and $k+m\le n$, there exists a $k$-to-$1$ mapping $$\phi: \UComp_{n,k,m} \to \cT_{n,k,m}.$$
Consequently, $\abs{\cT_{n,k,m}}=\frac{1}{k}\abs{\UComp_{n,k,m}}=\frac{1}{k}\abs{P_{n,k,m}}=w_{n,k,m}$.
\end{theorem}
\begin{proof}
Given an underlined composition $(A,c)$, we aim to construct a plane tree, say $T$, that will be the image of $(A,c)$ under $\phi$. We view each part $c_i$ as a ``claw'', i.e., a subtree with one root vertex attached by exactly $c_i$ children. Suppose $m=\sum_{1\le t\le i-1}c_t$ ($m=0$ if $i=1$), then these $c_i$ children are labeled as $m+1,m+2,\ldots,m+c_i$ from left to right, with those labels belonging to $A$ underlined. Now for each underlined label, we perform one ``amalgamation'', adjoining two components together, so that the initial $k$ claws that correspond to $c_1,c_2,\ldots,c_k$ eventually become a single tree after $k-1$ times of amalgamations. More precisely, suppose $x$ is the smallest underlined label in a certain component $a$, we find the ``next'' available component $b$, which could be simply a claw, or could be some tree produced from several previously conducted amalgamations. Attach 
$b$ to $a$ so that $x$ becomes the label of $b$'s root and we no longer view $x$ as underlined from now on. We view this process as one time amalgamation at label $x$. The reader is advised to use the concrete example in Fig.~\ref{map-phi} to go over the entire construction of $T$.

One crucial point to be made is that during this construction, we are treating these components essentially as cyclically listed, so there always exists the ``next'' component. For instance, if we perform the amalgamation at label $8$ in the underlined composition $\underline{1}\mid 23\underline{4}\mid 56\mid 7\underline{8}9$, the next component after $7\underline{8}9$ is understood to be $\underline{1}$, so that this $1$-claw $\underline{1}$ is attached to the middle child of the $3$-claw $7\underline{8}9$ for this amalgamation. This observation, on the other hand, explains the fact that $\phi$ is indeed a $k$-to-$1$ mapping, since for each cyclic shift of $c$, say $c'$, we can uniquely find another $(k-1)$-subset, say $A'$, such that after $k-1$ times of amalgamations, the two pairs $(A,c)$ and $(A',c')$ output the same plane tree (albeit with different labels). In other words, $\phi(A,c)=\phi(A',c')$; see Fig.~\ref{map-phi} for all four preimages of $T$ under the map $\phi$. Note that for general values of $n,k,m$, we might have $c'=c$ for a certain cyclic shift $c'$, but the underlined sets $A$ and $A'$ should still be distinct since $k$ and $k-1$ are coprime with each other. So indeed $(A,c)\neq(A',c')$. 

Conversely, given a plane tree $T\in\cT_{n,k,m}$, we label all of its non-root vertices $1,2,\ldots,n$ in a breadth-first fashion. I.e., label the first level nodes from left to right, then the second level nodes from left to right, so on and so forth. Then we underline and cut each non-root internal node, so that the original tree $T$ decomposes into $k$ claws. This should recover one out of $k$ preimages of $T$ under $\phi$. For the tree $T$ in Fig.~\ref{map-phi}, we recover the underlined composition $1\underline{2}3\mid \underline{4}\mid 56\underline{7}\mid 89$ in this way.
\end{proof}

The following is an example for the mapping $\phi$ in the case of $(n,k,m)=(9,4,3)$.
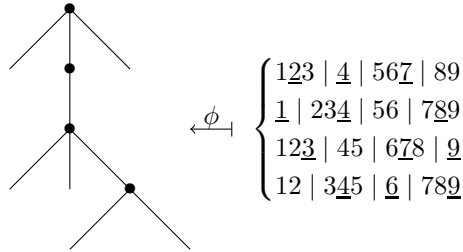
\begin{figure}[ht]
	\begin{tikzpicture}[scale=0.4]
		\centering
		\draw[-] (28,20.6) to (26,18.6);
		\draw[-] (28,20.6) to (28,18.6);
		\draw[-] (28,20.6) to (30,18.6);
		\draw[-] (28,18.6) to (28,16.6);
		\draw[-] (28,16.6) to (26,14.6);
		\draw[-] (28,16.6) to (28,14.6);
		\draw[-] (28,16.6) to (30,14.6);
		\draw[-] (30,14.6) to (28,12.6);
		\draw[-] (30,14.6) to (32,12.6);
		\node at (28,20.6){$\bullet$};
		\node at (28,18.6){$\bullet$};
		\node at (28,16.6){$\bullet$};
		\node at (30,14.6){$\bullet$};
		\draw [<-|] (32,16.6)--(33.4,16.6);
		\node at (32.7,17){$\phi$};
		\node at (38,16.6) {$\begin{cases}
		1\underline{2}3\mid \underline{4}\mid 56\underline{7}\mid 89\\
		\underline{1}\mid 23\underline{4}\mid 56\mid 7\underline{8}9\\
		12\underline{3}\mid 45\mid 6\underline{7}8\mid \underline{9} \\
		12\mid 3\underline{4}5\mid \underline{6}\mid 78\underline{9}
		\end{cases}$};
		% \node at (44,16.6) {$c=(3,1,3,2)\in [(3,2,3,1)], \text{ so } A=\{3,7,9\}\in \binom{[9]}{3}$};
		% \draw[-] (40.4,15.5) to (39.8,14.5);
		% \draw[-] (40.4,15.5) to (40.4,14.5);
		% \draw[-] (40.4,15.5) to (41,14.5);
		
		% \draw[-] (41.8,15.5) to (41.5,14.5);
		% \draw[-] (41.8,15.5) to (42.1,14.5);
		
		% \draw[-] (43.3,15.5) to (42.7,14.5);
		% \draw[-] (43.3,15.5) to (43.3,14.5);
		% \draw[-] (43.3,15.5) to (43.9,14.5);
		
		% \draw[-] (44.5,15.5) to (44.5,14.5);
		
		% \node at (39.8,14) {$1$};
		% \node at (40.4,14) {$2$};
		% \node at (41,14) {$\underline{3}$};
		% \node at (41.5,14) {$4$};
		% \node at (42.1,14) {$5$};
		% \node at (42.7,14) {$6$};
		% \node at (42.7,14) {$6$};
		% \node at (43.3,14) {$\underline{7}$};
		% \node at (43.9,14) {$8$};
		% \node at (44.5,14) {$\underline{9}$};
		
		%\node at (41,13.2) {$\triangle$};
		% \node at (43.3,13.2) {$\triangle$};
		% \node at (44.5,13.2) {$\triangle$};
	
	\end{tikzpicture}
	\caption{A plane tree $T$ and all four of its preimages under $\phi$}
	\label{map-phi}
\end{figure}

We proceed to present our second proof of theorem~\ref{thm:tree-2}. Let us first introduce four tree-related statistics. Given any plane tree $T$, we denote
\begin{align*}
\leaf(T) &:=\abs{\set{\text{the leaves of $T$}}},\\
\gleaf(T) &:=\abs{\set{\text{the good leaves of $T$}}},\\
\int(T) &:=\abs{\set{\text{the internal nodes of $T$}}}, \text{ and}\\
\pint(T) &:=\abs{\set{\text{the non-unitary internal nodes of $T$}}}.  
\end{align*}
Here the ``$p$'' in $\pint$ stands for ``prolific''.

After viewing the duality between the two interpretations of $w_{n,k,m}$ in terms of plane trees, one wonders if there exists a direct bijection defined on $\cT_n$, the set of plane trees with $n$ edges, such that the pair of tree statistics $(\leaf,\gleaf)$ corresponds to $(\int,\pint)$. Such a map would immediately imply theorem~\ref{thm:tree-2} in view of theorem~\ref{thm:tree-1}. As it turns out\footnote{We thank Zhicong Lin for bringing reference~\cite{LLZ2024} to our attention.}, an involution constructed in \cite{LLZ2024} gives rise to the following strengthening of this equidistribution. To make the current paper self-contained, we sketch a proof here. The reader is advised to check \cite[Thm.~2.2]{LLZ2024} for further details.

\begin{theorem}\label{thm:quadruple}
There exists an involution $\widetilde{\zeta}: \cT_n\to\cT_n$ such that for every $T\in\cT_n$,
\begin{align*}
(\leaf,\gleaf,\int,\pint)\:T &= (\int,\pint,\leaf,\gleaf)\:\widetilde{\zeta}(T).
\end{align*} 
\end{theorem}
\begin{proof}
Let $\cB_n$ be the set of all binary trees\footnote{I.e., plane trees where each internal node has either a left child, or a right child, or both.} with $n$ nodes. There is a natural bijection $\xi:\cT_n\to\cB_n$ that we are going to recall, and let $\zeta:\cB_n\to\cB_n$ be the map of mirror symmetry. Then $\widetilde{\zeta}$ is taken to be
$$\widetilde{\zeta}:=\xi^{-1}\circ \zeta \circ \xi,$$
which is clearly seen to be an involution defined over $\cT_n$. In a plane tree, nodes with the same parent are called {\it siblings} and the siblings to the left
(resp. right) of a node $v$ are called elder (resp. younger) siblings of $v$. For a plane tree $T\in \mathcal{T}_n$ , we define the binary tree $\xi(T)\in \mathcal{B}_n$  by requiring that for each pair of non-root nodes $(x,y)$ in $T$:
\begin{itemize}
\item [(a)]$y$ is the left child of $x$ in $\xi(T)$ only if when $y$ is the leftmost child of $x$ in $T$;
\item [(b)]$y$ is the right child of $x$ in $\xi(T)$ only if when $x$ is the closest elder sibling of $y$ in $T$.
\end{itemize}
It remains to show the quadruple of statistics $(\leaf,\gleaf,\int,\pint)$ is indeed transformed as claimed by $\widetilde{\zeta}$. Actually, there exists a one-to-one correspondence among nodes from a given plane tree that is denoted as $v\curly u$ in \cite{LLZ2024}. More precisely, for any node $v$ of a plane tree $T$, we can uniquely determine the node $u$ according to the following three cases.
\begin{itemize}
	\item If $v$ is an internal node, then $u$ is the youngest child of $v$.
	\item If $v$ is a leaf and no nodes in the path from $v$ to the root has elder siblings, then $u$ is the root $0$. We call $v$ a type I leave in this case.
	\item If $v$ is a leaf and $w$ is the first node that has elder sibling(s) in the path from $v$ to the root, then $u$ is the closest elder sibling of $w$. We call $v$ a type II leave in this case.
\end{itemize} 

Now it is routine to check the types of nodes under this correspondence. Namely, if $v$ is a type I leaf in $T$, then $u$, being the root $0$ in $T$, corresponds to the root $0$ in $\widetilde{\zeta}(T)$, which is an internal node. In particular, if $v$ is a good leaf of type I, then its parent cannot be $0$, which implies that the root of $\widetilde{\zeta}(T)$ has more than one child, hence it contributes to $\pint(\widetilde{\zeta}(T))$ as desired. If $v$ is a type II leaf in $T$, then $u$ is the parent of $v$ in $\widetilde{\zeta}(T)$, thus an internal node. And if in addition $v$ is a good leaf in $T$, then itself cannot have any elder siblings, which forces $u$ to have more than one child in $\widetilde{\zeta}(T)$. Finally, when $v$ is an internal node in $T$, then $u$ is its youngest child in $T$, meaning that $u$ has no right child in $\xi(T)$, thus no left child in $\zeta(\xi(T))$, making $u$ a leaf in $\widetilde{\zeta}(T)$. In particular, if $v$ is non-unitary in $T$, then $u$ has a closest elder sibling, say $w$, in $T$. This means $w$ has $u$ as its right child in $\xi(T)$, then $w$ has $u$ as its left child in $\zeta(\xi(T))$, making $u$ the leftmost child of $w$ in $\widetilde{\zeta}(T)$, i.e. a good leaf. So we see indeed, for the connected pair $v\curly u$, $v$ contributes to $\leaf(T)$ (resp.~$\gleaf(T)$, $\int(T)$, $\pint(T)$) if and only if $u$ contributes to $\int(\widetilde{\zeta}(T))$ (resp.~$\pint(\widetilde{\zeta}(T))$, $\leaf(\widetilde{\zeta}(T))$, $\gleaf(\widetilde{\zeta}(T))$).

% For a binary tree $B\in \mathcal{B}_n$ , let $\zeta(B)$ be the mirror symmetry of $B$.
% The involution $\widetilde{\zeta}:=\xi^{-1}\circ \zeta \circ \xi $. According to involution $\widetilde{\zeta}$ , we know that for a plane tree, the innernal nodes of $T\in \mathcal{T}_n$ correspond to the leaf nodes of $\widetilde{\zeta}(T)$. If a plan tree $T\in \mathcal{T}_n$, $v$ is the leaf of $T$. It can be divided into two cases. One is that all nodes in the path from $v$ to the root of $T$ have no elder siblings, and the corresponding node $u$ in $\widetilde{\zeta}(T)$ is the innernal point which only have one child. The other is that existing nodes in the path from $v$ to the root of $T$ have elder siblings, and the number of nodes from $v$ to the first node that has elder siblings when walking along the path from $v$ to the root of $T$ equals  the number of children of the corresponding node $u$ in $\widetilde{\zeta}(T)$. 	Thus, the leaves of the plane tree $T$ correspond to the innernal points of $\widetilde{\zeta}(T)$. The left-most child of a non-root vertex is a good leaf. Specially, if $y$ is the leftmost child of $x$ in $T$, then the corresponding node $u$ in $\widetilde{\zeta}(T)$ has children $x$ and $y$, so the corresponding node $u$ is deg$\ge 2$.
\end{proof}

%%%%%%%%%%%%%%%%%%%%%%%%%%%%%%%%%%%%%%%%%%%%%%%%%%%
\section{A group action for cyclic compositions}\label{sec:gamma}

A polynomial $f(x)=\sum_{i=0}^n a_i x^i$ is said to be {\it symmetric} if $a_i=a_{n-i}$ holds for every $0\le i\le n$. For the vector space consisted of all symmetric polynomials in $\bC[x]$ with degree no greater than $n$, one of its basis is easily seen to be given by $\set{x^k(1+x)^{n-2k}}_{0\le k\le \lfloor n/2 \rfloor}$. A notion stronger than symmetry stems from this consideration, namely the $\gamma$-positivity. A polynomial $f(x)=\sum_{i=0}^n a_i x^i$ is said to be {\it $\gamma$-positive} if it has an expansion
$$f(x)=\sum_{k=0}^{\lfloor\frac{n}{2}\rfloor}\gamma_kx^k(1+x)^{n-2k}$$ 
with $\gamma_k\ge 0$. The (shifted) Narayana polynomial $N_n(t)/t=\sum_{0\le m\le n-1}N_{n,m+1}t^m$ is known to be symmetric and $\gamma$-positive; see for example \cite[sect.~4.3]{Pet2015} and \cite{Ath2018}. 

Let $W_{n,k}(t):=\sum_{m=1}^k w_{n,k,m}t^m$. Now, thanks to the relations \eqref{id:w-Narayana} and \eqref{id:w-Narayana2}, one gets for free the symmetry and $\gamma$-positivity of the polynomials $W_{2k+1,k}(t)$ and $W_{2k-1,k}(t)$ from those of the Narayana polynomial. B\'ona et al. \cite[Remark 6.10]{BDL2022} raised a natural question of giving an alternative proof for the $\gamma$-positivities of $W_{2k+1,k}(t)$ and $W_{2k-1,k}(t)$. We aim to supply such a proof in this section, via a ``valley-hopping'' kind of group action defined over cyclic compositions. 

For the remainder of this section, we focus on the case of $n=2k+1$. Hence $\gcd(2k+1,k)=1$ and in particular, each equivalence class of cyclic compositions contains a unique representative, namely the dominant composition. All of our constructions from now on are actually done over $\DComp_k$, the set of dominant compositions of $2k+1$ into $k$ parts.

% \noindent {\bf [To do 3]:} Construct ``valley hopping'' on $\DComp_k$ to explain the $\gamma$-positivity of $N_k(t)$. As a byproduct, we get new interpretation of the $\gamma$ coefficients in terms of $2$-dominating compositions.

\begin{definition}
For $c=\left(c_1, c_2, \ldots, c_n\right) \in \DComp_n$, set $c_{n+1}:=1$. Then each part $c_i$, $1\le i\le n$ can be classified to be in one of the following four cases:
\begin{itemize}
	\item a non-unitary one if $c_i >1$ and $c_{i+1}=1$;
	\item a one non-unitary if $c_i =1$ and $c_{i+1}>1$;
	\item a double non-unitary if $c_i >1$ and $c_{i+1}>1$;
	\item a double one if $c_i =1$ and $c_{i+1}=1$.
\end{itemize}
\end{definition}
We denote $\Nuo(c)$, $\Onu(c)$, $\Dnu(c)$ and $\Do(c)$
respectively the set of parts $c_i$ in $c$ that belong to the above four cases, and let $\nuo(c)$ (resp. $\onu(c)$, $\dnu(c)$, $\sdo(c)$) be their corresponding cardinalities. Unless otherwise stated, the default value to be appended after each composition is $1$. We begin with a lemma that justifies the subsequent definition \ref{def:anchor}.

\renewcommand\theenumi{\roman{enumi}}
\renewcommand\labelenumi{(\theenumi)}
\begin{lemma}\label{lemma-ap}
Given a composition $c=\left(c_1, c_2, \ldots, c_n\right) \in \DComp_n$, we have for any $1\le i\le n$,
\begin{enumerate}
	\item if $c_i\in \Dnu(c)$, there exists a certain $j>i$ such that $f(c;i+1,j)=0$ and $c_{j+1}=1$;
	\item if $c_i\in \Do(c)$, there exists a certain $j<i$ such that  $f(c;j,i-1)>0$. 
\end{enumerate}
\end{lemma}
\begin{proof}
\begin{enumerate}
	\item First we note that by proposition~\ref{L-smooth}, the sequence $$f(c;1,1),\ldots,f(c;1,n)=1,f(c;1,n+1)=0$$ is L-smooth, and we see $f(c;1,i)\ge 1>f(c;1,n+1)$, so applying proposition~\ref{down-by-one} to this sequence we can find a certain index $j$, $i\le j<n+1$, such that $f(c;1,i)=f(c;1,j)$ (or equivalently $f(c;i+1,j)=0$) and $f(c;1,j+1)=f(c;1,j)-1$ (or equivalently $c_{j+1}=1$). Moreover, for the current case $c_i\in\Dnu(c)$ so $c_{i+1}>1$, rejecting the possibility of $j=i$.\\
  \item We have $c_i=1$ so $i>1$. Applying the characterization of dominating composition given by proposition~\ref{char:k-domi} we see that it suffices to take $j=1$.
\end{enumerate}
\end{proof}
\begin{definition}\label{def:anchor}
Let $c$ be a dominant composition. For each $c_i\in \Dnu(c)\cup \Do(c)$, we define its unique anchor point, denoted as $\ap(c_i)$, as follows:
\begin{enumerate}
	\item If $c_i\in \Dnu(c)$, find $c_j$ with the smallest $j > i$, such that $f(c;i + 1,j) = 0$ and $c_{j+1 } = 1$. Set $\ap(c_i):= c_j$. %and we denote the interval between $i$ and $j$ as $I_i$.
	\item If $c_i\in \Do(c),$ find $c_j$ with the largest $j < i$, such that $f (c,j,i-1) > 0$. Set $\ap(c_i):= c_j$. %and we denote the interval between $i$ and $j$ as $I_i$.
\end{enumerate}
\end{definition}
The classical Foata-Strehl action \cite[Chap.~4.1]{Ath2018} (a.k.a. the valley hopping) permutes the entries of a given permutation while keeping the values of all the entries. The action on dominant compositions that we are going to describe involves not only deletion and insertion of parts to change their positions, but also splitting and combining that will change their values. Consequently, to properly define our action $\psi_i$, we feel the need to associate with each composition $c\in\DComp_n$ a label sequence $l$, which is initially (i.e., before the action) taken to be $l=(1,2,\ldots,n)$ and remains a permutation of $[n]$. This label sequence is what the subindex $i$ in $\psi_i$ refers to. We use the following two-line array notation for such a pair:
\begin{align*}
\langle c|l \rangle=
\begin{pmatrix}
c_1 & c_2 & \cdots & c_n \\
l_1 & l_2 & \cdots & l_n
\end{pmatrix}.
\end{align*}

\begin{definition}\label{def-psi}
For any pair $\langle c|l\rangle$ with $c\in\DComp_n$ and $l$ being its label sequence, we define for each label $l_i\in[n]$ a map $\psi_{l_i}$ according to the following three cases, setting $\langle c'|l'\rangle:=\psi_{l_i}(\langle c|l\rangle)$:
% $c_i(1\le i\le n)$, where $c=\left(c_1, c_2, \ldots, c_n\right) \in \DComp_{n}$ and $l=(l_1,l_2,\cdots,l_n)$, we define valley hopping on $\DComp_n$ $\psi_{l_i}([c,l])=[\hat{c},\hat{l}]$ where $\hat{c}=\left (c_{1}^{\prime},c_{2}^{\prime},\ldots,c_{n}^{\prime}\right)$ and $\hat{l}=\left (l_{1}^{\prime},l_{2}^{\prime},\ldots,l_{n}^{\prime}\right)$ in the 
\begin{enumerate}
	\item if $c_i\in \Dnu(c)$ with $\ap(c_i)=c_j$, then $\langle c'|l'\rangle$ satisfies 
	\begin{align*}
		\begin{cases}
			c_t^{\prime}=c_t+c_{t+1}-1  &  \text{ if  $t=i$,} \\
			c_t^{\prime}=c_{t+1}& \text{ if $i+1\le t\le j-1$,}  \\
			c_t^{\prime}=1& \text{ if $t=j$, } \\
			c_t^{\prime}=c_t& \text{ otherwise,} 
		\end{cases}
		\text{ and } 
	\begin{cases}
		l_t^{\prime}=l_{t+1}  &  \text{ if $i\le t\le j-1$,}  \\
		l_t^{\prime}=l_{i}& \text{ if $t=j$,}  \\
		l_t^{\prime}=l_t& \text{ otherwise.} 
	\end{cases}
	\end{align*}
	\item if $c_i\in \Do(c)$ with $\ap(c_i)=c_j$ and $\alpha:=f(c;j,i-1)$, then $\langle c'|l'\rangle$ satisfies
	\begin{align*}
		\begin{cases}
			c_t^{\prime}=\alpha+1  &  \text{ if  $t=j$, }\\
			c_t^{\prime}=c_j-\alpha& \text{ if $t=j+1$,}  \\
			c_t^{\prime}=c_{t-1}& \text{ if $j+2\le t\le i$,}  \\
			c_t^{\prime}=c_t& \text{ otherwise,}
		\end{cases}
	\text{ and }
	\begin{cases}
			l_t^{\prime}=l_i  &  \text{ if $t=j$, } \\
			l_t^{\prime}=l_{t-1}& \text{ if $j+1\le t\le i$,}  \\
			l_t^{\prime}=l_t& \text{ otherwise. }
	\end{cases}
	\end{align*}	
	\item if $c_i\in\Nuo(c)\cup \Onu(c)$, then we set $\langle c'|l'\rangle=\langle c|l\rangle$.
\end{enumerate}
\end{definition}

\begin{example}
If $\langle c|l\rangle =\begin{pmatrix}
4 & 3 & 2 & 1 & 1\\
1 & 2 & 3 & 4 & 5
\end{pmatrix}$ with $c\in \DComp_{5,3}$, then $\psi_4(\langle c|l\rangle)=\langle c'|l'\rangle=
\begin{pmatrix}
4 & 2 & 2 & 2 & 1\\
1 & 4 & 2 & 3 & 5
\end{pmatrix}$ with $c'\in \DComp_{5,4}$. If $\langle c|l\rangle=
\begin{pmatrix}
5 & 3 & 1 & 1 & 1\\
1 & 2 & 3 & 4 & 5
\end{pmatrix}$ with $c\in \DComp_{5,2}$, then $\psi_1(\langle c|l\rangle)=\langle c'|l'\rangle=
\begin{pmatrix}
7 & 1 & 1 & 1 & 1\\
2 & 3 & 1 & 4 & 5
\end{pmatrix}$
with $c'\in \DComp_{5,1}$.
\end{example}

To prepare ourselves for the proof of the main result of this section, we collect below three lemmas concerning the mapping $\psi_{l_i}$.

\begin{lemma}\label{lem:psi-involution}
For each $i\in[n]$, $\psi_{l_i}$ induces a well-defined involution from $\DComp_n$ to itself, which we also denote as $\psi_{l_i}$. Moreover, suppose $c\in\DComp_{n,m}$. If $c_i\in\Dnu(c)$ then $\psi_{l_i}(c)\in\DComp_{n,m-1}$; if $c_i\in\Do(c)$ then $\psi_{l_i}(c)\in\DComp_{n,m+1}$.
\end{lemma}
\begin{proof}
We need to show that $\psi_{l_i}(c)\in\DComp_n$ and $\psi_{l_i}(\psi_{l_i}(c))=c$. There are three cases according to Definition~\ref{def-psi}. Case (iii) is quite clear since $\psi_{l_i}(c)=c$. 

For case (i) we have $c_i\in\Dnu(c)$ with $\ap(c_i)=c_j$. Clearly $c'=\psi_{l_i}(c)\in\Comp_{2n+1,n}$, we only need to show that $c'$ is dominating. For $1\le t< i$ or $j\le t\le n$, we have $f(c';1,t)=f(c;1,t)>0$, while for $i\le t<j$, we have $f(c';1,t)=1+f(c;1,t+1)>1$. So $c'\in\DComp_n$ by Proposition~\ref{char:k-domi}. Moreover, by our choice of the anchor point $c_j$, we must have $c_{j+1}=c'_{j+1}=1$, hence $c'_j=1\in\Do(c)$ with label $l'_j=l_i$. Consequently, when $\psi_{l_i}$ acts on $c'$, case (ii) of Definition~\ref{def-psi} applies and $\nonu(c')=\nonu(c)-1$. Using the extremity (i.e., smallest $j$ or largest $j$) in our definition of anchor point, it is not hard to check that $\ap(c'_j)=c'_i$, $f(c';i,j-1)=c_i-1+f(c;i+1,j)=c_i-1$, and $\psi_{l_i}(c')=c$ indeed.

Next for case (ii), we see $1=c_i\in\Do(c)$ with $\ap(c_i)=c_j$. We verify that $c'=\psi_{l_i}(c)$ is dominating. For $1\le t<j$ or $i\le t\le n$, we have $f(c';1,t)=f(c;1,t)>0$. Note that $\alpha=f(c;j,i-1)>0$ by our choice of the anchor point $c_j$, so 
$$f(c';1,j)=f(c;1,j-1)+(\alpha+1)-2=f(c;1,i-1)+c_i-2=f(c;1,i)>0.$$
For $j<t<i$, first note that $c_i=1$ forces $f(c;1,i-1)>1$. In addition, $f(c;t,i-1)\le 0$ due to the maximality of $j$ in our choice of the anchor point. So we see $f(c;1,t-1)=f(c;1,i-1)-f(c;t,i-1)>1$, hence 
\begin{align*}
f(c';1,t) &=f(c';1,j-1)+(c'_j-2)+(c'_{j+1}-2)+\cdots+(c'_t-2)\\
&=f(c;1,j-1)+(1-2)+(c_j-2)+\cdots+(c_{t-1}-2)\\
&=f(c;1,t-1)-1>0.
\end{align*}
This proves that $c'\in\DComp_n$. Note further that $c'_j=\alpha+1>1$ and $c'_{j+1}=c_j-\alpha=2-f(c;j+1,i-1)\ge 2$, which means that $c'_j\in\Dnu(c')$ with label $l'_j=l_i$, so that when $\psi_{l_i}$ acts on $c'$ case (i) applies and $\nonu(c')=\nonu(c)+1$. We omit the details of verifying $\ap(c'_j)=c'_i$ and $\psi_{l_i}(c')=c$ since they are similar to case (i).
\end{proof}

Next, we fix two labels $1\le a<b\le n$ and aim to show that $\psi_a$ commutes with $\psi_b$. Suppose $a$ and $b$ are the labels associated with the parts $c_i$ and $c_k$ respectively in $c\in\DComp_n$. Without loss of generality we assume that $i<k$. If $c_i$ or $c_k$ belongs to $\Nuo(c)\cup\Onu(c)$, it should be clear that $\psi_a(\psi_b(c))=\psi_b(\psi_a(c))$. For the remaining cases, we assume $\ap(c_i)=c_j$ and $\ap(c_k)=c_{\ell}$. We introduce the slightly more general notation $\langle x,y\rangle$ for the internal between $x$ and $y$ without specifying which is bigger. I.e., 
$$\langle x,y\rangle:=\set{z\in\bZ:\text{either $x\le z\le y$ or $y\le z\le x$}}.$$ 
% The interval $[i,j]$ (or $[j,i]$ if $c_i\in\Do(c)$) is denoted as $\langle i,j\rangle$, while $\langle k,\ell\rangle$ refers to the interval $[k,\ell]$ (or $[\ell,k]$ if $c_k\in\Do(c)$).

% Next we show that mapping $\psi$ is commutative if and only if either $\langle i,j\rangle \subseteq \langle k,\ell\rangle$, $\langle i,j\rangle \supseteq \langle k,\ell\rangle$ or $\langle i,j\rangle \cap \langle k,\ell\rangle=\varnothing$. 
\begin{lemma}\label{lemma-interval}
Given $c\in \DComp_n$, and $a,b,i,j,k,\ell$ as previously defined, we have either $\langle i,j\rangle \subseteq \langle k,\ell\rangle$, or $\langle i,j\rangle \supseteq \langle k,\ell\rangle$, or $\langle i,j\rangle \cap \langle k,\ell\rangle=\varnothing$.
\end{lemma}
\begin{proof}
We begin with several notations. Let
$$\tonu(c):=\set{c_t\in\Onu(c):i<t<k}$$ be the set of parts between $c_i$ and $c_k$ that are one non-unitary. For $1\le t\le n$, let $f_t:=f(c;1,t)$, and $\uf:=\min\set{f_d: c_d\in \tonu(c)}$. Without loss of generality, let us assume that $d\in (i,k)$ is the smallest index such that $f_d=\uf$.
% Let $\mathsf{ap}(a)=a^{\prime}$ and $\mathsf{ap}(b)=b^{\prime}$. 

We summarize in Table~\ref{tab:commute} the relation between $\langle i,j\rangle$ and $\langle k,\ell\rangle$ in various cases. We conduct a row-by-row verification of all these cases, which is tedious but for the most part straightforward. 	

\begin{enumerate}[font=\bfseries]
    \item[case 1]$\tonu(c)= \varnothing$. We consider the following four subcases.
    \begin{enumerate}
		\item $c_i,c_k\in\Dnu(c)$. This means that all parts between $c_i$ and $c_k$ are non-unitary, so we see $j>k$ since $c_{j+1}=1$. Moreover, $f(c;i+1,j)=0$ and $f(c;i+1,k)\ge 0$ implies that $f(c;k+1,j)\le 0$, thus $c_j$ is located to the right of (or is exactly) the anchor point of $c_k$, namely $c_{\ell}$. In other words, we have shown that $\langle i,j\rangle\supseteq \langle k,\ell\rangle$.
		\item $c_i,c_k\in\Do(c)$. This means that all parts between $c_i$ and $c_k$ are ones, so in particular $\ell<i$. Now $f(c;\ell,k-1)>0$ and $f(c;i,k-1)<0$ implies that $f(c;\ell,i-1)>0$. Therefore we must have $j\ge \ell$ and $\langle i,j\rangle\subseteq \langle k,\ell\rangle$.
		\item $c_i\in\Dnu(c)$, $c_k\in\Do(c)$ and $f_i>f_k$. This implies that $f(c;i+1,k)=f_k-f_i<0$ so the anchor point of $c_i$ sits strictly to the left of $c_k$, i.e., $j<k$. And $\tonu(c)=\varnothing$ so all parts between $c_j$ and $c_k$ are ones. Noting that $f(c;i+1,j)=0$ we can deduce $f(c;t,k-1)\le 0$ for all $i+1\le t\le k-1$. Hence the anchor point of $c_k$ sits to the left of $c_{i+1}$, meaning that $\ell\le i$ and $\langle i,j\rangle\subseteq \langle k,\ell\rangle$ as claimed.
		\item $c_i\in\Dnu(c)$, $c_k\in\Do(c)$ and $f_i\le f_k$. This means that $f(c;i+1,t)\ge 0$, for every $i+1\le t\le k$. This in turn implies that $j\ge k$. On the other hand, $f(c;i,k-1)-f(c;i+1,k)=c_i-c_k\ge 1$ thus $f(c;i,k-1)>0$. Consequently the anchor point of $c_k$ sits to the right of (or is exactly) $c_i$ and we have $\langle i,j\rangle\supseteq \langle k,\ell\rangle$.
	\end{enumerate}
	\item[case 2]$\uf< f_i$ and $\uf<f_k$. We consider the following four subcases.
	\begin{enumerate}	
		\item $c_i,c_k\in \Dnu(c)$. This implies that $f(c;i+1,d)=f_d-f_i< 0$, so $c_j$, the anchor point of $c_i$ sits strictly to the left of $c_d$, i.e., $j<d$. Therefore we have $i<j<d<k<\ell$ and $\langle i,j\rangle \cap \langle k,\ell\rangle=\varnothing$.
		\item $c_i,c_k\in \Do(c)$. This implies that $f(c;k,k)=f(c;d,d)=-1$, then we have $f(c;d,k-1)=f_{k-1}-f_{d-1}=(f_k+1)-(f_d+1)
		=f_k-f_d>0$. Thus $c_{\ell}$, the anchor point of $c_k$, is located to the right of (or is exactly) $c_d$, i.e., $d\le \ell$. Therefore we have $j<i<d\le \ell<k$ and $\langle i,j\rangle \cap \langle k,\ell\rangle=\varnothing$.
		\item $c_i\in\Dnu(c)$, $c_k\in\Do(c)$. This implies that $f(c;i+1,d)=f_d-f_i< 0$, so $c_j$ sits strictly to the left of $c_d$, i.e., $j<d$. And this also implies that $f(c;k,k)=f(c;d,d)=-1$, then we have $f(c;d,k-1)=f_{k-1}-f_{d-1}=(f_k+1)-(f_d+1)=f_k-f_d>0$, which means $d\le \ell$. Therefore we have $\langle i,j\rangle \cap \langle k,\ell\rangle=\varnothing$.
		\item $c_i\in\Do(c)$, $c_k\in\Dnu(c)$. We see $j<i<k<\ell$ directly from the definitions of $\Dnu(c)$ and $\Do(c)$. Hence $\langle i,j\rangle \cap \langle k,\ell\rangle=\varnothing$ as well.
	\end{enumerate}
	\item[case 3]$f_k=\uf<f_i$. We consider the following four subcases. 
	\begin{enumerate}
		\item $c_i,c_k\in \Dnu(c)$. We have $f(c;i+1,d)=f_d-f_i<0$, thus $c_j$ sits strictly to the left of $c_d$. Therefore we have $i<j<d<k<\ell$ and $\langle i,j\rangle\cap \langle k,\ell\rangle=\varnothing$. 
		\item $c_i,c_k\in \Do(c)$. We have $f(c;k,k)=-1$ and $f(c;d+1,k)=0$, which implies that $f(c;d+1,k-1)=f(c;d+1,k)-f(c;k,k)=1>0$. Thus $c_\ell$ sits strictly to the right of $c_d$, i.e., $d< \ell <k$. Recall that $c_i\in\Do(c)$, so we have $j<i<d< \ell<k$ and $\langle i,j\rangle\cap \langle k,\ell\rangle=\varnothing$.
		\item $c_i\in \Dnu(c), c_k\in \Do(c)$. With similar arguments as in (a) and (b), we deduce that $i<j<d< \ell<k$ and $\langle i,j\rangle\cap \langle k,\ell\rangle=\varnothing$.  
		\item $c_i\in \Do(c), c_k\in \Dnu(c)$. Simply by the definitions of $\Do(c)$ and $\Dnu(c)$ we see that $j<i<d<k<\ell$ and $\langle i,j\rangle\cap \langle k,\ell\rangle=\varnothing$. 
    \end{enumerate}
    \item[case 4]$f_k<\uf<f_i$. If $c_k\in\Dnu(c)$, then for the largest $m$ such that $c_m\in\tonu(c)$, we see $c_t>1$ for all $m<t\le k$, thus $f_m\le f_k$, in turn this means $\uf\le f_k$, a contradiction. So we must have $c_k\in\Do(c)$ and there are only two subcases to consider.
    \begin{enumerate}
    	\item $c_i,c_k\in \Do(c)$. This implies that $f(c;d+1,k-1)\le 0$. Due to the minimality of $f_d$, we have $f(c;t,k-1)\le 0$ for every $i+1\le t<k$. In addition, note that
    	\begin{align*}
    	f(c;i,k-1)&=f(c;i+1,k)+f(c;i,i)-f(c;k,k)\\
    	&=f(c;i+1,k)=f_k-f_i<0,
    	\end{align*}
    	we see $\ell<i$. Moreover, note that $f(c;\ell,i-1)=f(c;\ell,k-1)-f(c,i,k-1)>0$, which implies that $j\ge\ell$, and hence $\langle i,j\rangle\subseteq \langle k,\ell\rangle$.     	   	
    	\item $c_i\in \Dnu(c), c_k\in \Do(c)$. This implies that $f(c;i+1,d)=f_d-f_i<0$, so $j<d<k$. We argue in a similar fashion as in (a) to see that $f(c;t,k-1)\le 0$ for every $i+1\le t<k$. Therefore we have $\ell\le i$, meaning that $\langle i,j\rangle\subseteq \langle k,\ell\rangle$.  
    \end{enumerate}
    \item[case 5]$f_i \le \uf< f_k$. First note that $f_i\le\uf$ excludes the cases with $c_i\in\Do(c)$, so it suffices to consider the following two subcases.
    \begin{enumerate}
    	\item $c_i,c_k\in \Dnu(c)$. A moment of reflection reveals that $f(c;i+1,t)\ge 0$ for all $t\in [i+1,k]$, hence the anchor point of $c_i$ cannot lie in the inverval $[i+1,k]$, i.e., $j>k$. Now $f(c;i+1,j)=0$ and $f(c;i+1,k)=f_k-f_i>0$ lead to $f(c;k+1,j)=f(c;i+1,j)-f(c;i+1,k)<0$, implying that $k<\ell<j$. We have $\langle i,j\rangle\supseteq \langle k,\ell\rangle$.    	   	
    	\item $c_i\in \Dnu(c), c_k\in \Do(c)$. We argue as in (a) to deduce that $j>k$. Moreover with $f(c;i+1,k-1)=f(c;i+1,k)-f(c;k,k)>0$ we get $\ell\ge i+1$, meaning that $\langle i,j\rangle\supseteq \langle k,\ell\rangle$.     	
    \end{enumerate}
	\item[case 6] $\uf\ge f_i$ and $\uf\ge f_k$. Again $f_i\le\uf$ eliminates the cases with $c_i\in\Do(c)$. We discuss the remaining three subcases.
	\begin{enumerate}	
		\item $c_i,c_k\in\Dnu(c)$. With $\uf\ge f_k$ and $c_k\in\Dnu(c)$, the only possibility is $f_i\le \uf=f_k$. Applying a similar argument as in case 5(a), we derive that $\langle i,j\rangle\supseteq \langle k,\ell\rangle$.
		\item $c_i\in \Dnu(c)$, $c_k\in \Do(c)$ and $f_i> f_k$. With $f(c;i+1,k)=f_k-f_i<0$ we see $j<k$. Furthermore, note that $\uf\ge f_i>f_k$, so applying a similar argument as in case 4(b) we deduce that $\ell\le i$. Therefore we have $\langle i,j\rangle\subseteq \langle k,\ell\rangle$ as claimed.		   
		\item $c_i\in \Dnu(c)$, $c_k\in\Do(c)$ and $f_i\le f_k$. With $f(c;i+1,k-1)=f(c;i+1,k)-f(c;k,k)>0$ we deduce $i<\ell$. Knowing that $f_i\le f_k\le\uf$, we get $f(c;i+1,t)\ge 0$ for all $i+1\le t\le k$. Thus $j\ge k$, and we have $\langle i,j\rangle\supseteq \langle k,\ell\rangle$ as desired.
	\end{enumerate}
\end{enumerate}
\end{proof}

% \usepackage{multirow} 
\renewcommand{\arraystretch}{1.5}
\begin{table}
	\caption{Case-by-case breakdown of the relation between $\langle i,j\rangle$ and $\langle k,\ell\rangle$}\label{tab:commute}
	\centering 
	\resizebox{\textwidth}{!}{
\begin{tabular}{cc|c|c|cc|c}
	\hline
	\multicolumn{2}{c|}{\multirow{2}{*}{}} & \multirow{2}{*}{$c_i,c_k \in \mathsf{Dnu}(c) $} & \multirow{2}{*}{$c_i,c_k \in \mathsf{Do}(c) $} & \multicolumn{2}{c|}{$c_i\in \mathsf{Dnu}(c)$, $c_k\in \mathsf{Do}(c)$} & \multirow{2}{*}{$c_i\in \mathsf{Do}(c)$, $c_k\in \mathsf{Dnu}(c)$}  \\ \cline{5-6}
	\multicolumn{2}{c|}{}   &  &  & \multicolumn{1}{c|}{$f_i>f_k$} & $f_i\le f_k$ &    \\ \hline
	\multicolumn{2}{c|}{$\tonu(c)=\varnothing$}  & $\langle i,j\rangle\supseteq \langle k,\ell\rangle$ &$\langle i,j\rangle\subseteq \langle k,\ell\rangle$ & \multicolumn{1}{c|}{$\langle i,j\rangle\subseteq \langle k,\ell\rangle$}  & $\langle i,j\rangle\supseteq \langle k,\ell\rangle$ & inexistence \\ \hline
	\multicolumn{1}{c|}{\multirow{5}{*}{$\tonu(c)\ne \varnothing$}} & $\uf< f_i$ and $\uf<f_k$ &$\langle i,j\rangle \cap \langle k,\ell\rangle=\varnothing$& $\langle i,j\rangle \cap \langle k,\ell\rangle=\varnothing$ & \multicolumn{2}{c|}{$\langle i,j\rangle \cap \langle k,\ell\rangle=\varnothing$} & $\langle i,j\rangle \cap \langle k,\ell\rangle=\varnothing$ \\ \cline{2-7} 
	\multicolumn{1}{c|}{} & $f_k=\uf<f_i$ & $\langle i,j\rangle \cap \langle k,\ell\rangle=\varnothing$ & $\langle i,j\rangle \cap \langle k,\ell\rangle=\varnothing$ & \multicolumn{1}{c|}{$\langle i,j\rangle \cap \langle k,\ell\rangle=\varnothing$} & inexistence &$\langle i,j\rangle \cap \langle k,\ell\rangle=\varnothing$\\ \cline{2-7} 
	\multicolumn{1}{c|}{} & $f_k<\uf<f_i$ & inexistence & $\langle i,j\rangle\subseteq \langle k,\ell\rangle$ & \multicolumn{1}{c|}{$\langle i,j\rangle\subseteq \langle k,\ell\rangle$} & inexistence & \multirow{3}{*}{inexistence} \\ \cline{2-6}
	\multicolumn{1}{c|}{} & $f_i\le \uf<f_k$ & $\langle i,j\rangle\supseteq \langle k,\ell\rangle$ & inexistence & \multicolumn{1}{c|}{inexistence} & $\langle i,j\rangle\supseteq \langle k,\ell\rangle$ & \\ \cline{2-6}
	\multicolumn{1}{c|}{} & $\uf \ge f_i$ and $ \uf\ge f_k$ & $\langle i,j\rangle\supseteq \langle k,\ell\rangle$ & inexistence & \multicolumn{1}{c|}{$\langle i,j\rangle\subseteq \langle k,\ell\rangle$} & $\langle i,j\rangle\supseteq \langle k,\ell\rangle$ & \\ \hline
\end{tabular}	
}
\end{table}            

\begin{lemma}\label{lem:commute}
Given two labels $1\le a<b\le n$, the mapping $\psi_a$ commutes with $\psi_b$. I.e., we have $\psi_a(\psi_b(c))=\psi_b(\psi_a(c))$ for every $c\in\DComp_n$.
\end{lemma}
\begin{proof}
We assume the same notation as in Lemma~\ref{lemma-interval}. As already mentioned before Lemma~\ref{lemma-interval}, if $c_i$ or $c_k$ belongs to $\Nuo(c)\cup\Onu(c)$, then one of $\psi_a$ and $\psi_b$ becomes the identity map, so clearly $\psi_a(\psi_b(c))=\psi_b(\psi_a(c))$. We can now assume $\set{c_i,c_k}\subseteq\Dnu(c)\cup\Do(c)$. Thanks to Lemma~\ref{lemma-interval}, we know for the two intervals $\langle i,j\rangle$ and $\langle k,\ell\rangle$, either they are disjoint, or one contains the other. If they are disjoint, obviously $\psi_a(\psi_b(c))=\psi_b(\psi_a(c))$ holds. Otherwise one contains the other, we can use table~\ref{tab:commute} to verify the commutativity case-by-case. Since the arguments are mostly the same, we elaborate on one case and leave the others to the reader. 

Suppose we are in case 4(a), i.e., $c_i,c_k\in\Do(c)$ and $f_k<\uf<f_i$. Recall that we have shown $\ell\le j<i<k$ and recall Definition~\ref{def-psi}(ii). We observe that $f(c;i+1,k-1)=f(c;i+1,k)-f(c;k,k)=f_k-f_i+1\le 0$, which ensures that for $c':=\psi_a(c)$, we have respectively: 
\begin{align*}
f(c';t,k-1) &=f(c;t,k-1)\le 0,\text{ for $i+1\le t<k$ or $\ell<t\le j$,}\\
f(c';t,k-1) &=f(c';t,i)+f(c';i+1,k-1)\\
&=f(c;t-1,i-1)+f(c;i+1,k-1)\le 0, \text{ for $j+2\le t\le i$,}\\
f(c';j+1,k-1) &=f(c';j+1,i)+f(c';i+1,k-1)\\
&=f(c;j,i-1)-\alpha+f(c;i+1,k-1)\le 0, \text{ and}\\
f(c';\ell,k-1) &= f(c;\ell,k-1)>0.
\end{align*}
In other words, the anchor point for $c'_k=c_k=1$ in $c'$ is still $c'_{\ell}$ and $\psi_b(c')$ agrees with $\psi_b(c)$ outside of the interval $\langle i,j\rangle$ of indices. Consequently $\psi_a(\psi_b(c))=\psi_b(\psi_a(c))$ as claimed.
\end{proof}
                        
From lemmas~\ref{lem:psi-involution} and \ref{lem:commute}, we can conclude that all $\psi_i$'s are involutions and commute with each other. For any subset $S\subseteq [n]$ we can then define the map $\psi_S$: $\DComp_n \to \DComp_n$ by
$$ \psi_S(c):=\left(\prod_{i\in S}\psi_i\right)(c).$$
Hence the group $\bZ_2^n$ acts on $\DComp_n$ via the fuctions $\psi_S$, $S\subseteq [n]$. Here each element $g\in \bZ_2^n$ naturally corresponds to a subset $S(g)\subseteq [n]$. For instance, $(1,0,1,1,0)\in\bZ_2^5$ corresponds to the subset $\set{1,3,4}\subseteq [5]$. For any composition $c \in \DComp_n$, let $\Orb(c)=\set{\psi_{S(g)}(c): g \in \bZ_2^n} $ be the orbit of $c$ under this action; see Fig.~\ref{fig:orbit} below for a complete orbit containing the dominant composition $c=(6,1,1,1)\in\DComp_4$, corresponding to the term $t(1+t)^3$ in the $\gamma$-expansion:
$$N_4(t)=t+6t^2+6t^3+t^4=t(1+t)^3+3t^2(1+t).$$

\begin{figure}[htp]
	\begin{tikzpicture}[scale=0.8]
		%% nodes
		\draw(-.2,4.2) node{\text{\tiny $\begin{pmatrix}
		6 & 1 & 1 & 1\\
		1 & 2 & 3 & 4
		\end{pmatrix}$}};
		\draw(3.8,4.2) node{\text{\tiny $\begin{pmatrix}
		4 & 3 & 1 & 1\\
		3 & 1 & 2 & 4
		\end{pmatrix}$}};
		\draw(8,4.2) node{\text{\tiny $\begin{pmatrix}
		3 & 2 & 3 & 1\\
		4 & 3 & 1 & 2
		\end{pmatrix}$}};
		\draw(12,4.2) node{\text{\tiny $\begin{pmatrix}
		3 & 2 & 2 & 2\\
		4 & 3 & 2 & 1
		\end{pmatrix}$}};
		\draw(3.8,8.2) node{\text{\tiny $\begin{pmatrix}
		5 & 2 & 1 & 1\\
		2 & 1 & 3 & 4
		\end{pmatrix}$}};
		\draw(8.2,8.2) node{\text{\tiny $\begin{pmatrix}
		4 & 2 & 2 & 1\\
		3 & 2 & 1 & 4
		\end{pmatrix}$}};
		\draw(3.7,.2) node{\text{\tiny $\begin{pmatrix}
		3 & 4 & 1 & 1\\
		4 & 1 & 2 & 3
		\end{pmatrix}$}};
		\draw(8,.2) node{\text{\tiny $\begin{pmatrix}
		3 & 3 & 2 & 1\\
		4 & 2 & 1 & 3
		\end{pmatrix}$}};
		
		%% edges
		\draw[-latex] (0,4.8)->(4,7.7);
		\draw[-latex] (0,3.7)->(4,0.8);
		\draw[-latex] (8,7.7)->(12,4.8);
		\draw[-latex] (8,0.8)->(12,3.7);
		\draw[-latex] (1,4.25)->(2.5,4.25);
		\draw[-latex] (5,4.25)->(6.7,4.25);
		\draw[-latex] (9.3,4.25)->(10.8,4.25);
		\draw[-latex] (5,8.25)->(7,8.25);
		\draw[-latex] (5,0.25)->(6.7,0.25);
		\draw[-latex] (5,4.8)->(7,7.7);
		\draw[-latex] (5,0.8)->(7,3.7);
		\draw[-latex] (5,7.7)->(7,0.8);
		
		%% labels
		\draw(2,6.7) node{$\psi_2$};
		\draw(2,4.55) node{$\psi_3$};
		\draw(2,1.8) node{$\psi_4$};
		\draw(6,8.5) node{$\psi_3$};
		\draw(6,-.05) node{$\psi_2$};
		\draw(6.4,4.55) node{$\psi_4$};
		\draw(10,4.55) node{$\psi_2$};
		\draw(5.3,1.8) node{$\psi_3$};
		\draw(7,1.8) node{$\psi_4$};
		\draw(10.1,6.7) node{$\psi_4$};
		\draw(10,1.8) node{$\psi_3$};
		\draw(6.5,6.4) node{$\psi_2$};
	\end{tikzpicture}
	\caption{A complete orbit under the action $\psi$}\label{fig:orbit}
\end{figure}
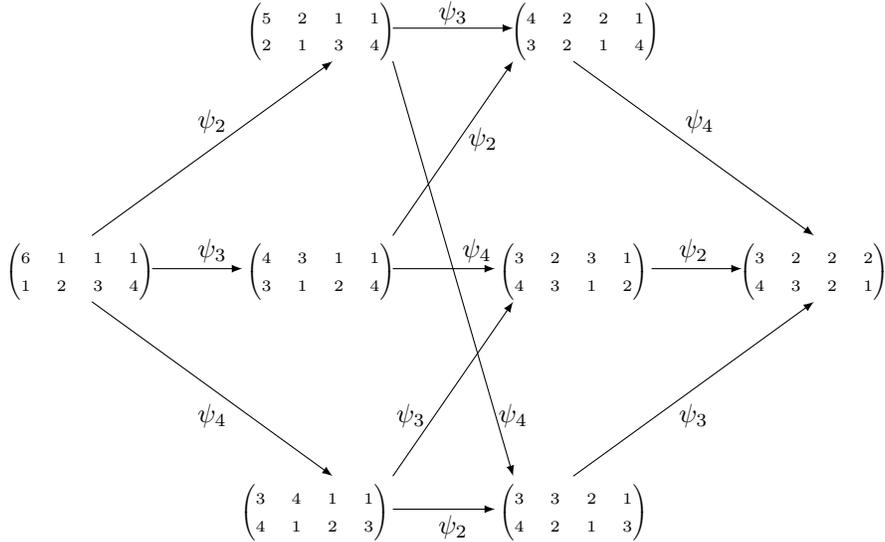

We are in a position to give a ``valley-hopping'' proof of the following result, which includes two new interpretations for the $\gamma$-coefficients of Narayana polynomials in terms of dominant compositions (or equivalently, cyclic compositions).
\begin{theorem}
The Narayana polynomial $N_k(t)$ is $\gamma$-positive for all $k \ge 1$. It has the following $\gamma$-expansion:
\begin{align}
N_k(t)=\sum_{j=1}^{\left \lfloor \frac{k+1}{2} \right \rfloor } \gamma^{N}_{k,j} t^j(1+t)^{k+1-2j} ,
\end{align}
where
\begin{align}
\label{gamma-inter1}
\gamma^{N}_{k,j}&=\abs{\set{c\in \DComp_{k,j} : \Dnu(c)=\varnothing}}\\
&=\abs{\set{c\in \DComp_{k,k-j+1} : \Do(c)=\varnothing}}.\label{gamma-inter2}
\end{align}
\end{theorem}
\begin{proof}
Suppose a composition $c\in\DComp_k$ satisfies $\nuo(c)=j$. Since $c$ is dominant, it must begin with a non-unitary part and recall our convention $c_{k+1}=1$. Each part in $\Nuo(c)$ is followed by a one, hence we have $1\le j\le \lfloor\frac{k+1}{2}\rfloor$ and $\onu(c)=j-1$. Now for the remaining $k-j-(j-1)=k-2j+1$ parts of $c$, take one of them, say $c_i$ with label $l_i$. It could be either $c_i\in\Dnu(c)$, then $\psi_{l_i}(c)$ is a composition with one more double one part and one fewer double non-unitary part than $c$, witnessing a switch from $t$ to $1$ in the factor $(1+t)$; or it could be $c_i\in\Do(c)$, then $\psi_{l_i}(c)$ is a composition with one more double non-unitary and one fewer double one than $c$, witnessing a switch from $1$ to $t$ in the factor $(1+t)$. In summary, we conclude that the $\bZ^k_2$-action divides the set $\DComp_k$ into disjoint orbits, with each orbit $\Orb(c)$ corresponding to a term $t^j(1+t)^{k+1-2j}$, where $j=\nuo(c)$. 

Now the two interpretations represent two extreme (or uniform) choices for the representative of each orbit. Namely, we can either ``hop'' every double one left to become a double non-unitary, resulting in the interpretation \eqref{gamma-inter2}; or we can ``hop'' every double non-unitary right to become a double one and gives rise to the interpretation \eqref{gamma-inter1}.

% Graphically, there is a unique compositions in each orbit which has no double non-unitary. Now, let $\bar{c}$ be this unique element in $\mathsf{Orb}(c)$, then $\mathsf{do}(c)=n-\mathsf{nuo}(c)-\mathsf{onu}(c)$ and $\nonu(c)=\mathsf{nuo}(c)=\mathsf{onu}(c)+1$. 
% In fact, for any $c\in \DComp_k$ and $c_{n+1}=1$, $c$ can be divided into several fragments beginning and ending with non-unitary and the last one beginning with non-unitary and ending with 1. For any other $c^{\prime}\in \mathsf{Orb}(c)$, it can be obtained from $\bar{c}$ by repeatedly applying $\psi_i$ for some double one $i$ of $\bar{c}$. Each time this happens, $\nonu$ increases by 1 and $\mathsf{do}$ decreases by 1. Thus
% $$\sum_{[d,m]\in \mathsf{Orb}([c,l])}t^{\nonu(d)}=t^{\nonu(\bar{c})}(1+t)^{\mathsf{do}(\bar{c})}=t^{\nonu(\bar{c})}(1+t)^{k+1-2\mathsf{nu}(\bar c)}.$$
% The second equality follows similarly.
\end{proof} 
\begin{remark}
The $\gamma$-coefficient $\gamma^N_{k,j}$ is known to have an explicit expression:
\begin{align}\label{gammaN:formula}
\gamma^N_{k,j}=\frac{1}{k}\binom{k}{j}\binom{k-j}{j-1},
\end{align}
which can be proved directly via our new interpretation given in \eqref{gamma-inter1} or \eqref{gamma-inter2}. We sketch a proof in terms of \eqref{gamma-inter1} here. Let $\beta_{k,j}$ be the number of compositions $c\in\DComp_{k,j}$ such that $\Dnu(c)=\varnothing$ and $c_k=1$. We claim that
\begin{enumerate}
	\item $\gamma^N_{k,j}=\beta_{k,j}+\beta_{k-1,j-1}$, and
	\item $\beta_{k,j}=\frac{1}{k}\binom{k}{j-1}\left(\binom{k-j+1}{j}-\binom{k-j-1}{j-2}\right)$.
\end{enumerate}
Then a straightforward computation with binomial coefficients immediately yields \eqref{gammaN:formula}. Now (i) follows from observing that either $c_k=1$ or $c_k=2$ for any dominant composition $c\in\DComp_k$. We remove $c_k$ in the latter case to arrive at a composition counted by $\beta_{k-1,j-1}$. While for (ii), note that $\beta_{k,j}$ also counts the number of cyclic compositions $[c]\in\CComp_{2k+1,k,j}$ such that for any cyclic shift $c'$ of $c$, we have $\Dnu(c')=\varnothing$. Such cyclic compositions can be enumerated via a standard combinatorial argument, like associating them with restricted solutions to
$$x_1+x_2+\cdots+x_k=2k+1$$
in $\bZ_{>0}$, analogous to the proof of Proposition~\ref{prop:comp}. The details are omitted here.
\end{remark}

% \begin{corollary}
% For all $k\ge 1$, $ w_{2k+1,k}(t)$ is $\gamma$-positive,
% \begin{align}
% w_{2k+1,k}(t)&=\sum_{j=1}^{\left \lfloor \frac{k+1}{2} \right \rfloor  } \binom{2k+1}{k-1} \frac{(k-1)!}{(k-2j+1)!(j-1)!j!}  t^{j}(1+t)^{k+1-2j}\label{eq1}\\&=\sum_{j=1}^{\left \lfloor \frac{k+1}{2} \right \rfloor  } \gamma _{k,j}^{w} t^{j}(1+t)^{k+1-2j},\label{eq2}
% \end{align}

% \end{corollary}
% It answers a question in B\'ona et al. 's article\cite{BDL2022}, giving a direct combinatorial proof of equation(\ref{eq1}) in the vein of “valley-hopping”.

Since the factor $\binom{2k+1}{k-1}$ (or $\binom{2k-1}{k-1}$) in \eqref{id:w-Narayana} (resp.~\eqref{id:w-Narayana2}) is independent of the parameter $m$, our group action $\psi$ can be easily lifted to the set of plane trees, and thus leads to the following interpretation. 

% For the next result, we cyclically order all the internal nodes in a given plane tree. I.e., we start with the root, then traverse all remaining internal nodes from left to right in a breadth-first fashion, and the last internal node is viewed as preceding the root.
\begin{corollary}
For all $k\ge 1$, we have the expansion
$$W_{2k+1,k}(t)=\sum_{j=1}^{\lfloor\frac{k+1}{2}\rfloor}\gamma^{W_1}_{k,j}t^j(1+t)^{k+1-j},$$ 
where
\begin{align*}
\gamma^{W_1}_{k,j}&=\binom{2k+1}{k-1}\gamma^N_{k,j}\\
&=\abs{\left \{ (A,[c]): A \in\binom{[2k+1]}{k-1},c\in \DComp_{k,j}, \dnu(c)=0\right\}}\\
&=\abs{\left \{ (A,[c]): A \in\binom{[2k+1]}{k-1},c\in \DComp_{k,k-j+1}, \sdo(c)=0\right\}},
\end{align*}
and for all $k\ge 2$,
$$W_{2k-1,k}(t)=\sum_{j=1}^{\lfloor\frac{k+1}{2}\rfloor}\gamma^{W_2}_{k,j}t^j(1+t)^{k+1-j},$$ 
where
\begin{align*}
\gamma^{W_2}_{k,j}&=\binom{2k-1}{k-1}\gamma^N_{k-1,j}\\
&=\abs{\left \{ (A,[c]): A \in\binom{[2k-1]}{k-1},c\in \DComp_{k-1,j}, \dnu(c)=0\right\}}\\
&=\abs{\left \{ (A,[c]): A \in\binom{[2k-1]}{k-1},c\in \DComp_{k-1,k-j}, \sdo(c)=0\right\}}.
\end{align*}
\end{corollary}
\begin{proof}
Recall the $k$-to-$1$ mapping $\phi$ in Theorem~\ref{thm:bij-tree-pair} and notice that for the current case, $n=2k+1$ is coprime to $k$, so the factor $1/k$ can be applied directly to the composition $c$, turning $k$ pairs $(A,c')$, $c'\in[c]$ into a single pair $(A,[c])$, which then corresponds to a unique plane tree\footnote{One way to make the correspondence between $(A,[c])$ and a plane tree unique is as follows. We agree to always find the unique cyclic shift, say $\hat{c}$ of $c$ such that $\hat{c}$ is dominant, then label the children of $\hat{c}_1$-claw as $1,2,\ldots,\hat{c}_1$, the children of $\hat{c}_2$-claw as $\hat{c}_1+1,\ldots,\hat{c}_1+\hat{c}_2$, etc.}. Moreover, the restrictions we place on the representative, i.e., the unique dominant composition $c\in[c]$ are inherited directly from \eqref{gamma-inter1} and \eqref{gamma-inter2}. This establishes the expansion for $W_{2k+1,k}(t)$ and the two interpretations of $\gamma^{W_1}_{k,j}$. The results for $W_{2k-1,k}(t)$ and $\gamma^{W_2}_{k,j}$ follow analogously.
\end{proof}

\subsection*{Acknowledgements}

This work was supported by the National Natural Science Foundation of China and the Natural Science Foundation Project of Chongqing.
%grant 12171059  (No.~cstc2021jcyj-msxmX0693).
\bibliographystyle{amsplain}

\end{document}